\documentclass[10pt, a4paper]{amsart}
\usepackage{amsmath}
\usepackage{amssymb}
\usepackage{amsthm}
\usepackage{graphicx}
\usepackage{url}
\usepackage{enumerate}
\usepackage{xcolor}
\usepackage{mathrsfs} 
\usepackage{t1enc}
\usepackage{lineno}

\newtheorem{thm}{Theorem}
\newtheorem{lem}[thm]{Lemma}

\newtheorem{cor}[thm]{Corollary}

\theoremstyle{definition}
\newtheorem{defn}[thm]{Definition}
\newtheorem{rem}[thm]{Remark}
\newtheorem{ex}[thm]{Example}

\DeclareMathOperator{\diam}{diam}
\DeclareMathOperator{\dist}{dist}

 \newcommand{\eps}{\varepsilon}

 \newcommand{\Y}{\mathbb{Y}}
 
 \newcommand{\Al}{\mathcal{A}}

 \newcommand{\Lom}{\mathcal{L}}

 \newcommand{\N}{\mathbb{N}}

 \newcommand{\R}{\mathbb{R}}

 \newcommand{\set}[1]{\left\{#1\right\}}

\newcommand{\dbar}{\bar{d}}
 
\DeclareMathOperator{\Emp}{\mathfrak{m}} 

\DeclareMathOperator{\di}{\overrightarrow d}
\DeclareMathOperator{\m}{\mathfrak m}

\DeclareMathOperator{\Gen}{Gen}
\newcommand{\seq}[2]{\{#1_{#2}\}_{#2=0}^{\infty}}
\newcommand{\emp}[2]{\mathfrak{m}(\underline{#1},#2)}

\newcommand{\s}[1]{\underline{#1}=\{#1_n\}_{n=0}^{\infty}}
\newcommand{\und}[1]{\underline{#1}}

\newcommand{\bbd}[2]{D_B({#1},{#2})}

\newcommand{\om}[1]{\hat{\omega}({#1})}

\renewcommand{\phi}{\varphi}

\DeclareMathOperator{\M}{\mathcal{M}}
\DeclareMathOperator{\MT}{\mathcal{M}_T}

\DeclareMathOperator{\G}{\mathcal{G}}

\makeatletter
\def\blfootnote{\gdef\@thefnmark{}\@footnotetext}
\makeatother

\author{Dominik Kwietniak \and Martha {\L}{\k{a}}cka \and Piotr Oprocha}

\address[D. Kwietniak]{
Faculty of Mathematics and Computer Science, Jagiellonian University in Krakow, ul. \L o\-jasiewicza 6, 30-348 Krak\'ow, Poland -- \and --
Institute of Mathematics, Federal University of Rio de Janeiro,  Av. Athos da Silveira Ramos 149, Centro de Tecnologia - Bloco C, Cidade Universit\'{a}ria,
Ilha do Fund\~{a}o C.P. 68530, Rio de Janeiro 21945-909, Brazil
}\email{dominik.kwietniak@uj.edu.pl}
\urladdr{www.im.uj.edu.pl/DominikKwietniak/}

\address[M. {\L}{\k{a}}cka]{
Faculty of Mathematics and Computer Science, Jagiellonian University in Krakow, ul. \L o\-jasiewicza 6, 30-348 Krak\'ow, Poland}\email{martha.lacka@doctoral.uj.edu.pl}
\urladdr{www2.im.uj.edu.pl/MarthaLacka/}

\address[P. Oprocha]{AGH University of Science and Technology, Faculty of Applied
	Mathematics, al.
	Mickiewicza 30, 30-059 Krak\'ow, Poland
	-- \and --
National Supercomputing Centre IT4Innovations, Division of the University of Ostrava,
Institute for Research and Applications of Fuzzy Modeling,
30. dubna 22, 70103 Ostrava, Czech Republic}
\email{oprocha@agh.edu.pl}

\title{Generic points for dynamical systems with average shadowing}
\date{\today}

\begin{document}
\begin{abstract} We study the Besicovitch pseudometric $D_B$ for compact dynamical systems. The set of generic points of ergodic measures turns out to be  closed with respect to $D_B$. 
It is proved that the weak specification property implies the average asymptotic shadowing property and the latter property does not imply the former one nor the almost specification property. Furthermore an example of a proximal system with the average shadowing property is constructed.
It is proved that to every invariant measure $\mu$ of a compact dynamical system one can associate a certain asymptotic pseudo orbit such that any point asymptotically tracing in average that pseudo orbit is generic for $\mu$. A simple consequence of the theory presented is that every invariant measure has a generic point in a system with the asymptotic average shadowing property. 
\end{abstract}
\blfootnote{\textup{2010} \textit{Mathematics Subject Classification}: 37A05, 37B05, 37B10}
\maketitle

The specification property or one of its variants are often applied in the theory of dynamical systems\footnote{By a dynamical system we mean a pair $(X,T)$, where $X$ is a compact metric space and $T$ is a continuous map from $X$ to itself.} to construct invariant measures with special properties (for example, measures of maximal entropy or equilibrium states, see \cite{Bowenbook,DGSbook}). These properties also help to define points with predetermined statistical behaviour like generic points or ``non-typical'' points, see \cite{BS00,DGSbook}.
Here, using the Besicovitch pseudometric, we show how to construct a certain asymptotic pseudo orbit associated with a given measure $\mu$.
This generalizes several proofs of existence of generic points for a given invariant measure, provided that a sufficiently strong tracing method is available in the system, e.g. see \cite{Dateyama81, Dateyama82,Sigmund70, Sigmund74, Sigmund76, Sigmund77}.

The Besicovitch pseudodistance measures the average separation of points along two sequences in the space of all infinite $X$-valued sequences $X^{\infty}=\{\seq{x}{n}\,:\,x_n\in X\}$ and is given for $\s{x}$, $\underline{x}'=\seq{x'}{n}\in X^{\infty}$ by
	\[
	\bbd{\underline x}{\underline x'}=\limsup_{N\to\infty}\frac 1N\sum_{n=0}^{N-1}\rho(x_n,x_n').
	\]
It extends naturally to a pseudometric on $X$ given by the average separation of points
along orbits of a dynamical system $(X,T)$, that is,
\[
D_B(x,y)=\limsup_{N\to\infty}\frac{1}{N}\sum_{n=0}^{N-1}\rho(T^n(x),T^n(y)),
\]
where $x,y\in X$ and $\rho$ is a metric on $X$. An equivalent pseudometric is well known in the ergodic theory and symbolic dynamics as \emph{$\dbar$-pseudometric}, see \cite{ShieldsBook, WeissBook}. A similar Weyl pseudodistance and the accompanying notion of quasi-unifrom convergence were also widely studied, see \cite{DI,JK}.

We prove that a point is generic for a non-necessarily ergodic measure if it asymptotically traces a special approximate trajectory 
(which in turn is a variant of the average asymptotic pseudoorbit), where the Besicovitch pseudodistance is used to define the pseudotrajectory and measures the accuracy of tracing. Recall that $x\in X$ is generic for $\mu$ if
\[
\lim_{n\to\infty}\frac{1}{n}\sum_{j=0}^{n-1} f(T^j(x))=\int_X f\:\text{d}\:\!\mu \quad\text{ for every continuous }f\colon X\to\R.
\]
As a corollary we get that if a system has the asymptotic average shadowing property (a notion introduced by Gu in \cite{Gu07}, see below), then every invariant measure has a generic point, because in such a system every average asymptotic pseudo-orbit is followed by the orbit of some  point. The same result is proved independently in \cite{DTY} via a direct construction. We believe that using the Besicovitch pseudodistance provides a new perspective and leads to short proofs extending the theory (see the proof of Theorem~\ref{thm:pertrace} and Example \ref{ex:example}).

In order to compare our result with earlier developments, recall that generic points have full measure for every ergodic invariant measure, but non-ergodic invariant measures need not to have generic points. Sigmund proved that if a dynamical system has the periodic specification property (see \cite{Sigmund74}), then every invariant measure has a generic point. Many authors generalized this result.

We show that yet another variant of specification, known as the weak specification property implies the asymptotic average shadowing property. This together with results of \cite{KKO} gives us the following diagram illustrating the connections between main generalizations of specification (we refer the reader to \cite{KLO} for more details; here ``AASP'' stands for ``(the) asymptotic average shadowing property''):
\begin{figure}[h]
\includegraphics[scale=0.7]{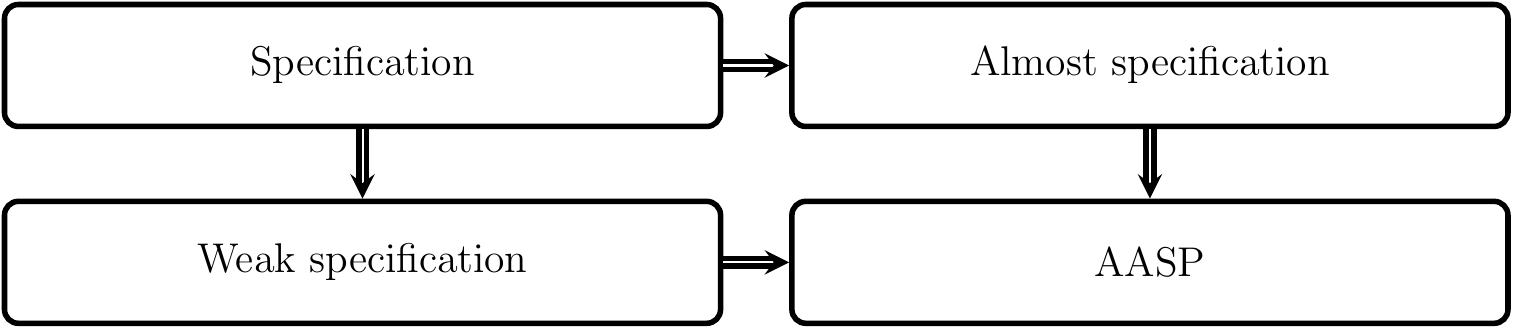}\label{fig:ex}
\end{figure}

It is proved in \cite{KOR} that there are systems with the weak specification property, but without the almost specification property, hence there are systems with the asymptotic average shadowing property, but without almost specification property. This answers a problem we considered in \cite{KKO}. We finish the paper with an example demonstrating that there are proximal systems with the average shadowing property, which also answers some natural questions one may ask about properties of systems with a generalized shadowing. 

This paper is organized as follows. In Section 1 we recall some standard definitions from topological dynamics and related fields.
Section 2 discusses the definition and basic properties of the Besicovitch pseudometric.
In Section 3 we show that a set of measures generated by an orbit varies continuously as a function from $X$ equipped with the Besicovitch pseudometric to the space of compact subsets of invariant measures considered with the Hausdorff distance. It follows that the set of generic points of ergodic measures is a closed set with respect to the Besicovitch pseudodistance. Section 4 establishes the relation between weak specification property and asymptotic average shadowing property. Possible applications of our studies on Besicovitch pseudodistance are provided in Section 5. Section 6 contains an example of a proximal, mixing shift space exhibiting the average shadowing property, but neither almost nor weak specification.

\section{Topological dynamics}

\subsection*{Notation} Throughout this paper a pair $(X,T)$ is a dynamical system.
This means that $X$ is a compact metric space and $T\colon X\to X$ is a continuous map.
All properties of $T$ considered here are invariants of topological conjugacy, hence they do not depend on the choice of metric.
Therefore we fix a metric $\rho$ for $X$ and assume without loss of generality that the diameter of $X$ is equal to $1$.

Given a compact metric space $Y$ and $A\subset Y$ we denote by $\dist(y,A)$ the \emph{distance of a point $y\in Y$ to the set $A$}.
For $\eps>0$ we write $A^{\eps}$ for the \emph{$\eps$-hull} of $A$, that is the set $\{y\in X\,:\,\dist(y,A)<\eps\}$.
Let $2^{Y}=\{K\subset Y\,:\,K\text{ compact and }K\neq\emptyset\}$ and $H_{Y}\colon2^{Y}\times2^{Y}\to\mathbb R_+$ be the \emph{Hausdorff metric} on $2^Y$ given by
\[
H_Y(A,B)=\max\big\{\inf\{\eps>0\,:\,B\subset A^{\eps}\},\,\inf\{\eps>0\,:\,A\subset B^{\eps}\}\big\}.
\]
The \emph{upper asymptotic density} of $A\subset \N_0$ is as usual denoted by $\dbar(A)$ and given by
\[
\dbar(A)=\limsup_{n\to\infty}\frac{\#(A\cap \{0,1,\ldots,n-1\})}{n}.
\]
If the above upper limit is in fact a limit, we denote it by $d(A)$ and call the \emph{asymptotic density} of $A$.
\subsection*{Recurrence properties} Recall that $T$ is \emph{transitive} if for all non-empty open sets $U,V\subset X$ there exists
$n\geq 0$ such that $T^n(U)\cap V\neq\emptyset$. We say that $T$ is \emph{(topologically) mixing} if for all non-empty open sets $U,V\subset X$ there exists
$N\geq 0$ such that for all $n\geq N $ one has $T^n(U)\cap V\neq\emptyset$. A map $T$ is \emph{(topologically) weakly mixing} if the product map $T\times T$ is transitive. Finally, $T$ is \emph{chain mixing} if for every $\delta>0$ and all $x,y\in X$  there exists $N\in\N$ such that for every $n\geq N$ there is a finite sequence $x_0,x_1,\ldots,x_n$ in $X$ satisfying $x_0=x$, $x_n=y$, and $\rho(T(x_{i-1}),x_i)<\delta$ for $i=1,\ldots,n$. 

\subsection*{Average shadowing} A sequence $\seq{x}{n}\in X^{\infty}$ is a \emph{$\delta$-average pseudoorbit} for $T$ if there is an integer $N_0>0$ such that for every $N>N_0$ and $k\geq 0$  one has
\[
\frac{1}{N}\sum_{n=0}^{N-1}\rho\big(T(x_{n+k}),x_{n+k+1}\big)<\delta.
\]
A sequence $\seq{x}{n}\in X^{\infty}$ is an \emph{asymptotic average pseudoorbit} for $T$ if
\[
\lim\limits_{N\to\infty}\frac{1}{N}\sum_{n=0}^{N-1} \rho(T(x_n),x_{n+1})=0.
\]
 Given $\eps>0$ we say that a sequence $\{x_j\}_{j=0}^{\infty}$ is \emph{$\eps$-shadowed in average} by a point $x\in X$ if
$$\limsup\limits_{N\to\infty}\frac{1}{N}\sum_{n=0}^{N-1}\rho\big(T^n(x), x_n\big)<\eps.$$
We say that $\seq{x}{n}\in X^{\infty}$ is \emph{asymptotically shadowed in average} by $x\in X$ if
$$\lim\limits_{N\to\infty}\frac{1}{N}\sum_{n=0}^{N-1}\rho\big(T^n(x),x_n\big)=0.$$
A dynamical system $(X,T)$ has the \emph{average shadowing property} if for every $\eps>0$ there exists $\delta>0$ such that every $\delta$-average pseudo-orbit is $\eps$-shadowed in average by some point. 
A system $(X,T)$ has the \emph{asymptotic average shadowing property} 
if every asymptotic average pseudo-orbit of $T$ is asymptotically shadowed in average by some point.

\subsection*{Weak specification}
An \emph{orbit segment} of a point $x$ over an interval $[a,b]$, where $a\le b$ are nonnegative integers is the sequence
$$T^{[a,b]}(x)=(T^a(x), T^{a+1}(x),\ldots, T^{b-1}(x),T^b(x)).$$
Let $N\colon \N\to\N$ be any function.
A~family of orbit segments $\xi=\{T^{[a_j,b_j]}(x_j)\}_{j=1}^n$ is an \emph{$N$-spaced specification} if
$a_i-b_{i-1} \ge  N(b_i-a_i+1) $ for $2\le i \le n$.
A specification $\xi=\{T^{[a_j,b_j]}(x_j)\}_{j=1}^n$
is $\eps$-traced by $y\in X$ if
\[
\rho(T^k(y),T^k(x_i)) \le \eps \quad\text{for } a_i\le k \le b_i\text{ and }1\le i \le n.
\]
A  dynamical system $(X,T)$ has the \emph{weak specification property} if for every $\eps>0$ there is a nondecreasing function $M_\eps \colon \N \to \N$ with $M_\eps(n)/n\to 0$ as $n\to \infty$ such that any $M_\eps$-spaced specification is $\eps$-traced by some point in $X$. We refer the reader to \cite{KLO} for more information about specification.

\subsection*{Invariant measures}
Let $\mathcal M(X)$ be the set of all Borel probability measures on $X$. Let $\di\colon\mathcal M(X)\times\mathcal M(X)\to\mathbb R_+$ be the Prokhorov metric, that is let
\[
\di(\mu,\nu)=\inf\big\{\eps>0\,:\,\mu(B)\leq\nu(B^{\eps})+\eps\text{ for every Borel set }B\subset X\big\}.
\]
It is well known that the topology induced by $\di$ is the weak$^\ast$ topology on $\mathcal M(X)$.
A measure $\mu\in\mathcal M(X)$ is:\begin{itemize}
\item \emph{$T$-invariant}, if $\mu(T^{-1}(B))=\mu(B)$ for every Borel set $B\subset X$,
\item \emph{ergodic}, if it is $T$-invariant and for every Borel set $B\subset X$ with $T^{-1}(B)=B$ either $\mu(B)=0$ or $\mu(B)=1$.
\end{itemize}
We denote by $\mathcal M_T(X)\subset\mathcal M(X)$ the family of all $T$-invariant measures and by $\mathcal M_T^e(X)$ the family of all ergodic measures. The \emph{measure center} of $(X,T)$ is the smallest closed subset $F$ of $X$ such that $\mu(F)=1$ for every $T$-invariant measure $\mu$.

For $x\in X$ let $\hat\delta(x)\in\mathcal M(X)$ be the~Dirac measure supported on~$\{x\}$.
Given $\underline{x}=\seq{x}{j}\in X^{\infty}$ and $n\in\N$ we define an \emph{empirical measure}:
$$\m(\underline x, n)=\frac{1}{n}\sum_{i=0}^{n-1}\hat\delta(x_i).$$
For $x\in X$ let $\m_T(x,n)$ denote the measure $\m(\{T^j(x)\}_{j=0}^{\infty},n)$. If $T$ is clear from the context, then we write $\m(x,n)$ for $\m_T(x,n)$. We call such a measure an \emph{empirical measure of $x$}.
A point $x\in X$ is called \emph{generic} for $\mu\in\mathcal M(X)$ if the sequence $\{\m_T(x,n)\}_{n=0}^{\infty}$ converges to $\mu$ in the weak$^\ast$ topology. The set of generic points of $\mu$ is denoted $\Gen(\mu)$.

A measure $\mu\in \M(X)$ is a \emph{distribution measure} for a sequence $\s{x}\in X^\infty$ if $\mu$ is a limit of some subsequence of
$\{\Emp(\underline{x},n)\}_{n=1}^\infty$. The set of all distribution measures of a sequence $\underline{x}$ is denoted by $\om{\underline{x}}$. Clearly, $\om{\underline{x}}$ is a closed and nonempty subset of $\M(X)$ (since $\M(X)$ is compact), it is also connected, because $\di(\Emp(\underline{x},n),\Emp(\underline{x},n+1))\to 0$ as $n\to\infty$.

\subsection*{Symbolic dynamics}
Let $\mathcal{A}$ be a finite set of \emph{symbols} with discrete topology. Let $\mathcal{A}^\infty$ be a product topological spaces.
By $\sigma$ we denote the continuous onto map from $\mathcal{A}^\infty$ to itself given by $\sigma(\omega)_i=\omega_{i+1}$, where $\omega=(\omega_i)\in\mathcal{A}^\infty$. A \emph{shift space} is a nonempty, closed, and $\sigma$-invariant subset of $\mathcal{A}^\infty$.
For more details on symbolic dynamics we refer to \cite{ISDC}.




\section{Besicovitch Pseudometric}\label{sec:Bes}


To elucidate the construction of generic points we introduce the \emph{Besicovitch pseudometric}. 
Recall that a pseudometric on $X$ is a function $p\colon X\times X\to[0,\infty)$ which is symmetric, obeys the triangle inequality,
and $p(x,x)=0$ for all $x\in X$. Every pseudometric space is a topological space
in a natural manner, and most notions defined for metric spaces, like the (uniform) equivalence of metrics, remain meaningful for
pseudometric spaces.

Recall that $X^\infty$ denotes the set of all sequences indexed with nonnegative integers with entries in $X$.
Given $T\colon X\to X$ we can define a map (also denoted by $T$) on $X^\infty$ by $T(\underline x)=\{T(x_n)\}_{n=0}^{\infty}$
for any $\s{x}\in X^{\infty}$. We define the \emph{shift map} on $X^\infty$ by
$\sigma(\underline x)_n=x_{n+1}$ for $n\in\N_0$ and $\s{x}\in X^{\infty}$.

We can restate the definition of a pseudo orbit in our notation as follows:
a sequence $\s{x}\in X^\infty$ is a $\delta$-\emph{pseudo orbit} if $D_{\sup}(T(\underline{x}),\sigma(\underline{x}))<\delta$,
where $D_{\sup}$ denotes the supremum metric on $X^\infty$. The orbit of $y\in X$ is said to \emph{$\eps$-trace}
$\underline{x}\in X^\infty$ if the sequence of $\{T^n(y)\}_{n=0}^{\infty}\in X^\infty$ is $\eps$-close to $\underline{x}$ in
the $D_{\sup}$ sense. It turns out that in order to construct invariant measures and generic points it is convenient to work
with the following pseudometric on $X^\infty$.

\begin{defn}
	The \emph{Besicovitch pseudometric} $D_B$ on $X^{\infty}$ is defined for $\s{x}$, $\underline{x}'=\seq{x'}{n}\in X^{\infty}$ by
	\[
	\bbd{\underline x}{\underline x'}=\limsup_{N\to\infty}\frac 1N\sum_{n=0}^{N-1}\rho(x_n,x_n').
	\]
	The \emph{Besicovitch pseudometric} on $(X,T)$, also denoted by $D_B$,  equals the $D_B$-distance between orbits, that is,
	\[
	D_B(x,x')=D_B(\{T^n(x)\}_{n=0}^{\infty},\{T^n(x')\}_{n=0}^{\infty}).
	\]
\end{defn}
It is easy to see that these functions are indeed pseudometrics on $X^\infty$ and $X$, respectively. Furthermore,
$D_B$ is $\sigma$-invariant on $X^\infty$ and $T$-invariant on $X$. Besicovitch used a similar pseudodistance in
his study of almost periodic functions, but the idea of measuring the average distance along two orbits in a
dynamical systems is so natural that probably it has been reinvented independently by many authors, see Aulsander \cite{Auslander59}
who mentions earlier works of Fomin and Oxtoby, or \cite{BFK, Downarowicz14, DownarowiczGlasner, FGJ, OW04} for more recent applications.


The pseudometric $D_B$ depends on $\rho$ but this is not reflected in our notation because we will show that equivalent metrics induce uniformly equivalent
Besicovitch pseudometrics. 


\begin{lem}\label{lem:db-prim}
Let $X$ be a compact metric space and $\rho$ be a compatible metric on $X$. Then the Besicovitch pseudometric and the pseudometric
$D'_B$ on $X^\infty$ given for $\und{x}=\{x_i\}_{i=0}^\infty$ and $\und{z}=\{z_i\}_{i=0}^\infty$ by
\[
D'_B(\und{x},\und{z})=\inf
\{\delta>0: \dbar(\{n\ge 0: \rho(x_n,z_n)\ge \delta\})<\delta\}
\]
are uniformly equivalent on $X^\infty$.
\end{lem}

\begin{proof} Given $\und{x},\und{z}\in X^\infty$ and $\delta>0$ we define
\[
J_\delta(\und{x},\und{z})=\{n\ge 0:\rho(x_n,z_n)\ge \delta\}.
\]
Let $\chi_\delta$ denote the characteristic function of $J_\delta(\und{x},\und{z})$.
Recall that $\diam X=1$, hence
we have
\[
\delta\cdot\chi_\delta(n)\le \rho(x_n,z_n)\le 
\chi_\delta(n)+\delta.
\]
Summing the above from $n=0$ to $N-1$, dividing by $N$ and passing with $N$ to infinity we get
\begin{equation}\label{star}
\delta\cdot \dbar(J_\delta(\und{x},\und{z}))\le D_B(\und{x},\und{z})\le 
\dbar(J_\delta(\und{x},\und{z}))+\delta.
\end{equation}

Note also that
\begin{equation}\label{starstar}
D'_B(\und{x},\und{z})\leq \eps\text{ if and only if }\dbar(J_\eps(\und{x},\und{z}))<\eps.
\end{equation}

Fix $\eps>0$. We claim that if $\delta<\eps/ 2$, 
then $D'_B(\und{x},\und{z})<\delta$ implies $D_B(\und{x},\und{z})<\eps$.
To prove the claim assume that $D'_B(\und{x},\und{z})< \eps/2$ 
and use \eqref{starstar} to see that
the right hand side of \eqref{star} is strictly smaller than $\eps$. This proves our claim and yields uniform continuity of $\textrm{id}\colon (X^\infty, D_B')\to (X^\infty,D_B)$.

We now show that $\textrm{id}\colon (X^\infty, D_B)\to (X^\infty,D_B')$ is uniformly continuous. Fix $\eps>0$ and $\delta<\eps^2$.
Take any pair $\und{x},\und{z}\in X^\infty$ such that $D_B(\und{x},\und{z})<\delta$. Then by the first inequality in \eqref{star} and the definition of $\delta$ we see that
\[
\eps\cdot \dbar(J_\eps(\und{x},\und{z})) \le D_B(\und{x},\und{z})<\eps^2.
\]
Therefore $\dbar(J_\eps(\und{x},\und{z}))<\eps$, which implies that  $D'_B(\und{x},\und{z})\leq \eps$ by \eqref{starstar}.
\end{proof}

Lemma \ref{lem:db-prim} yields immediately the following corollary (for a proof it is enough to note that if $\rho$ and $\tilde\rho$ are uniformly equivalent, then so are $D'_B$ and $\tilde D'_B$).
\begin{cor}\label{cor:equiv_m}
Let $\tilde\rho$ be another compatible metric on $X$ and $\tilde D_B$ and $\tilde D'_B$ be defined as above with $\tilde\rho$  in place of $\rho$. Then
$\tilde D_B$ and $D_B$ are uniformly equivalent on $X^\infty$.
\end{cor}

We will now prove that if $(X,\sigma)$ is a shift space (symbolic dynamical system),
then the Besicovitch pseudometric is uniformly equivalent with the $\dbar$-pseudometric
measuring the upper asymptotic density of the set of coordinates at which two symbolic sequences differ.
This is probably a folklore, but for completeness we present a proof inspired by \cite{DI}.

Let $(X,T)$ be a dynamical systems and $x\in X$. We define $\und{x}_T$ to be the sequence $x, T(x), T^2(x),\ldots\in X^\infty$.
Given a continuous function $f\colon X\to\R$ we define
\[
f(\und{x}_T)=(f(x),f(T(x)),\ldots,f(T^n(x)),\ldots)\in \R^\infty
\]
and for any family $\mathcal{F}$ containing continuous functions from $X$ to $\mathbb{R}$ we set
\[
D^{\mathcal{F}}_B(x,x')=\sup_{f\in\mathcal{F}} D_B(f(\und{x}_T),f(\und{x}'_T)).
\]
\begin{thm}\label{thm:familyF}
Let $(X,T)$ be a dynamical system and let $\mathcal{F}$ be a uniformy equicontinuous and uniformly bounded family of real-valued functions on $X$ such that
$\mathcal{F}_T=\{f\circ T^j : f\in\mathcal{F},\, j\ge 0\}$ separates the points of $X$. Then the pseudometrics $D_B$ and $D^\mathcal{F}_B$ are uniformly equivalent on $X$.
\end{thm}
\begin{proof}
Using compactness it is easy to see that if the family $\mathcal{F}_T$ separates the points of $X$, then there exists a sequence $\{f_n\}_{n=1}^\infty\subset \mathcal{F}_T$ separating the points of $X$, where $f_n=g_n\circ T^{k(n)}$ for some $k(n)\ge 0$ and $g_n\in \mathcal{F}$.
Let $x,x'\in X$ and $n\ge 1$. Define $z=T^{k(n)}(x)$ and $z'=T^{k(n)}(x')$.
Note that
\[
D_B(f_n(\und{x}_T),f_n(\und{x}'_T))=
D_B(g_n(\und{z}_T),g_n(\und{z}'_T))=D_B(g_n(\und{x}_T),g_n(\und{x}'_T)).
\]
Therefore
\begin{align*}
\sup_{n\ge 1} D_B(f_n(\und{x}_T),f_n(\und{x}'_T))&=\sup_{n\ge 1} D_B(g_n(\und{x}_T),g_n(\und{x}'_T))\\
&\le D^{\mathcal{F}}_B(x,x')=\sup_{f\in\mathcal{F}} D_B(f(\und{x}_T),f(\und{x}'_T)).
\end{align*}
Take any sequence of weights $\{a_n\}_{n=1}^\infty\subset$ such that
\[
0<a_n<1 \quad\text{and}\quad\sum_{n=1}^\infty a_n ||f_n||_\infty<\infty,
\]
where $||\cdot||_\infty$ denotes the sup norm. For $y,z\in X$ define
\[
\rho'(y,z)=\sum_{n=1}^\infty a_n |f_n(y)-f_n(z)|.
\]
Note that $\rho'$ is a metric on $X$ equivalent to $\rho$ because $f_n$'s are equicontinuous and separate the points of $X$.

Fix $\eps>0$. Find $m\in\N$ such that
\[
\sum_{n=m+1}^\infty a_n ||f_n||_\infty<\eps/2.
\]
Let $x,x'\in X$ be such that $D^{\mathcal{F}}_B(x,x')<\eps/(2m)$.
By Corollary \ref{cor:equiv_m} the Besicovitch pseudometric $D'_B$ associated with $\rho'$ is uniformly equivalent to $D_B$.
Therefore it is enough to check that $D_B'(x,x')<\eps$. We have
\begin{align*}
D'_B(x,x')&=\limsup_{N\to\infty}\frac{1}{N}\sum_{i=0}^{N-1}\rho'(T^i(x),T^i(x'))
\\&=\limsup_{N\to\infty}\frac{1}{N}\sum_{i=0}^{N-1}\sum_{n=1}^\infty a_n |f_n(T^i(x))-f_n(T^i(x'))|
 \\&\le\limsup_{N\to\infty}\frac{1}{N}\sum_{i=0}^{N-1}\sum_{n=1}^m \big(a_n |f_n(T^i(x))-f_n(T^i(x'))|+\eps/2\big)
 \\&=\eps/2 + \limsup_{N\to\infty}\frac{1}{N}\sum_{i=0}^{N-1}\sum_{n=1}^m a_n |f_n(T^i(x))-f_n(T^i(x'))|
 \\&\le\eps/2 + \sum_{n=1}^m a_n \left(\limsup_{N\to\infty}\frac{1}{N}\sum_{i=0}^{N-1}|f_n(T^i(x))-f_n(T^i(x'))|\right)<\eps.
\end{align*}
The last inequality holds because we took $x,x'\in X$ with $D^{\mathcal{F}}_B(x,x')<\eps/(2m)$. This proves that
\[
\textrm{id}\colon (X, D^\mathcal{F}_B)\to (X,D_B)\]
is uniformly continuous.

To prove that
\[
\textrm{id}\colon (X, D_B)\to (X,D^\mathcal{F}_B)\]
is uniformly continuous one can follow the reasoning as in the proof of Lemma \ref{lem:db-prim} to conclude that thanks to the uniform boundedness of $\mathcal{F}$ it is enough to show that for every $\eps>0$ there exists $\delta>0$ such that if $x,x'\in X$ satisfy $D_B(x,x')<\delta$ then for every $f\in\mathcal{F}$ we have
\[
\dbar(\{n: \rho(f(T^n(x)),f(T^n(x')))\ge \eps\})<\eps.
\]
The latter easily follows from uniform equicontinuity of $\mathcal{F}$ and Lemma \ref{lem:db-prim}.
\end{proof}
Observe that any finite family $\mathcal{F}$ is uniformly equicontinuous and uniformly bounded.
In particular, if $X$ is the space of sequences over a finite alphabet and $T$ is the shift map, then taking $\mathcal{F}=\{\iota_0\}$, where $\iota_0(x_0x_1x_2\ldots)=x_0$ yields the following corollary.
\begin{cor}\label{cor:uniformlyequiv}
If $\mathcal{A}$ is a finite set with discrete topology, then the Besicovitch pseudometric associated with any admissible metric on
the product space $\mathcal{A}^\infty$ and the
\emph{$\dbar$-pseudometric} on $\mathcal{A}^\infty$ given by
\[
\dbar(x,x')=\dbar(\{n\ge 0 : x_n\neq x'_n\})=\limsup_{n\to\infty}\frac{\#\{0\le j<n : x_j\neq x'_j\}}{n}
\]
are uniformly equivalent.
\end{cor}

\section{Besicovitch Pseudometric and invariant measures}\label{sec:Bes-inv}

We are going to show that the map $\underline{x}\mapsto \om{\underline{x}}$ is a uniformly continuous map from $X^\infty$ equipped with $D_B$ to the space of nonempty compact subsets of $\M(X)$ equipped with the Hausdorff metric.

\begin{lem}\label{lem:ciagloscbesic}
For every $\eps>0$ there is $\delta>0$ such that if $\underline{x},\underline{x}'\in X^{\infty}$ and $\bbd{\underline x}{\underline x'}<\delta$, then
$\di(\m(\und{x},n),\m(\und{x}',n))\le \eps$ for all sufficiently large $n$.
\end{lem}
\begin{proof}Fix $\eps>0$.
By Lemma \ref{lem:db-prim} there exist $\delta>0$ such that if $\underline{x},\underline{x}'\in X^{\infty}$ and $\bbd{\underline x}{\underline x'}<\delta$, then one can find $N\in\N$ such that for every $n\geq N$ one has
\[
	\#\big\{0\leq j< n\,:\,\rho(x_j,x_j')\geq \eps\big\}<n\cdot \eps.
\]
Take $n\ge N$ and a Borel set $B\subset X$. We have
	\begin{align*}
		\emp{x}{n}(B)=\frac{1}{n}\#\big\{0\leq j< n \,:\,x_j\in B\big\}
&\le \m(\underline x',n)(B^{\eps})+\eps.
	\end{align*}
	Hence $\di\big(\m(\underline x,n),\m(\underline x',n)\big)\leq\eps$ for all $n\ge N$.
\end{proof}
\begin{thm}\label{thm:ciagloscbesic}
	For every $\eps>0$ there is $\delta>0$ such that for all $\underline{x},\underline{x}'\in X^{\infty}$ with $\bbd{\underline x}{
		\underline x'}<\delta$ one has $H_{\mathcal{M}(X)}(\om{\underline x},\om{\underline x'})<\eps$.
\end{thm}
\begin{proof} Fix $\eps>0$. Use Lemma \ref{lem:ciagloscbesic} to find $\delta$ for $\eps/3$.
Take $\und{x},\und{x}'\in X^\infty$ such that $D_B(\und{x},\und{x}')<\delta$.
Let $\mu\in\hat\omega(\und{x})$. Then $\mu=\lim\limits_{k\to\infty}\{\emp{x}{n_k}\}_{k=0}^{\infty}$ for some strictly increasing sequence $\seq{n}{k}\in\N^{\infty}$.  Let $\nu$ be an accumulation point of the sequence $\{\mathfrak m(\underline x',n_k)\}_{k=0}^{\infty}$.
Taking a subsequence we can assume that
\[
\lim\limits_{k\to\infty}\{\mathfrak m(\underline x',n_k)\}_{k=0}^{\infty}=\nu.
\]
By Lemma \ref{lem:ciagloscbesic} we have $\di(\emp{x}{n_k},\mathfrak m(\underline x',n_k))\le \eps/3$ provided that $k$ is large enough.
Let $k$ be sufficiently large to guarantee also that
$\di(\mu,\emp{x}{n_k})<\eps/3$ and $\di(\nu,\mathfrak m(\underline x',n_k))<\eps/3$.
Consequently,
	$$\di(\mu,\nu)\leq\di(\mu,\emp{x}{n_k})+\di(\emp{x}{n_k},\mathfrak m(\underline x',n_k))+\di(\mathfrak m(\underline x',n_k),\nu)<\eps.$$
Hence we have
\[
\dist\left(\mu, \om{\underline x'}\right)=\min\left\{\di(\mu,\nu)\,:\,\nu\in\om{\underline x'}\right\}<\eps.
\]
Therefore $\max\left\{\dist(\mu, \om{\underline x'})\,:\,\mu\in\om{\underline x}\right\}<\eps$.
 By the same argument we also have $\max\left\{\dist(\nu, \om{\underline x})\,:\,\nu\in\om{\underline x'}\right\}<\eps$ and hence $H_{\mathcal M(X)}\big(\om{\underline x},\om{\underline x'}\big)<\eps$.
\end{proof}
\begin{cor}\label{cor:D_B-zero}Let $\und{x},\und{x}'\in X^\infty$.
	If $\bbd{\underline x}{\underline x'}=0$, then $\om{\underline x}=\om{\underline x'}$.
\end{cor}
\begin{rem}
It is easy to see that if $\s{x}\in X^{\infty}$ is a $\delta$-average-pseudo-orbit for $T$, then $\bbd{T(\underline x)}{\sigma(\underline x)}<\delta$.
\end{rem}
For any continuous map $T\colon X\to X$ let $\hat T\colon\mathcal M(X)\to\mathcal M(X)$ be given by the formula $\hat T(\mu)=\mu\circ T^{-1}$.
The proof of the following fact can be found in \cite[Proposition 3.2]{DGSbook} (see also \cite[Theorem~6.7]{Walters}).
\begin{thm}
	The map $\hat T\colon \mathcal M(X)\to\mathcal M(X)$ is continuous. Additionally, $\hat T$ is surjective if and only if $T$ is surjective.
\end{thm}
\begin{cor}\label{cor:im:51}
	For every $\eps>0$ there is $\delta>0$ such that if $\s{x}\in X^\infty$ fulfills $D_B(T(\und{x}),\sigma(\und{x}))<\delta$, 
then for every $\mu\in\om{\underline x}$ one has $\di(\mu,\hat T(\mu))<\eps$.
\end{cor}
\begin{proof}
Let $\mu\in\om{\underline x}$ and $\m(\underline x, k_n)\to \mu$ as $n\to\infty$, where $\{k_n\}_{n=0}^\infty$ is a strictly increasing sequence of nonnegative integers. For every Borel set $B$ and every $x\in X$ one has $\hat\delta(x)\big(T^{-1}(B)\big)=\hat\delta\big(T(x)\big)(B)$. Consequently,
	\begin{align*}		\hat T\big(\m(\underline x, n)\big)(B)=\m(\underline x, n)\left(T^{-1}(B)\right)&= \frac 1n\sum_{j=0}^{n-1}\hat{\delta}(x_j)\left(T^{-1}(B)\right)\\
		&=\frac 1n\sum_{j=0}^{n-1}\hat\delta\big( T(x_j)\big)(B)=\m\big(T(\underline x), n\big)(B).
	\end{align*}
	Because of the continuity of $\hat T$ one has
\begin{equation}\label{eq:raz}
\hat T(\mu)=\hat T\big(\lim\limits_{n\to\infty}\m(\underline x, k_n)\big)=\lim\limits_{n\to\infty}\hat T\big(\m(\underline x, k_n)\big)=\lim\limits_{n\to\infty}\m\big(T(\underline x),k_n\big).
\end{equation}
Note that $\di\big(\m(\underline x,n),\m(\sigma(\underline x),n)\big)\leq 2/n$ holds for every $n\in\N$. Hence
\begin{equation}\label{eq:dwa}
\mu=\lim\limits_{n\to\infty}\m(\sigma(\underline x),k_n).
\end{equation}
To finish the proof fix $\eps>0$ and use Lemma~\ref{lem:ciagloscbesic} to pick $\delta>0$ for $\eps/2$. Then
$D_B\big(\sigma(\underline x), T(\underline x)\big)<\delta$
implies $\di(\m(\sigma(\underline x),k_n),\m(T(\underline x),k_n))\le \eps/2$ for all sufficiently large $n$.
Therefore $\di\big(\mu, \hat{T}(\mu)\big)<\eps$ by \eqref{eq:raz} and \eqref{eq:dwa}.
\end{proof}
\begin{rem}
	A sequence $\s{x}$ is an asymptotic average pseudo orbit if and only if $\bbd{T(\underline x)}{\sigma(\underline x)}=0$.
\end{rem}
\begin{cor}
	If $\s{x}\in X^{\infty}$ is an  asymptotic average pseudo orbit, then $\om{\underline x}\subset\mathcal M_T(X)$.
\end{cor}
\begin{proof}
	Since $\bbd{T(\underline x)}{\sigma(\underline x)}=0$ it follows from Corollary~\ref{cor:im:51} that for every $\eps>0$ and every $\mu\in\om{\underline x}$ one has $\di(\mu, \hat T(\mu))\leq\eps$. Hence $\mu=\hat T(\mu)$ and so $\mu\in\mathcal M_T(X)$.
\end{proof}
 We finish this section with a yet another application of the Besicovitch pseudometric. Although we will not use this result we add it here since we belive it is of independent interest. Note that Weiss \cite{WeissBook} proves the analog of Theorem \ref{thm:generic-dbar-closed} for symbolic system and convergence with respect to the $\dbar$-pseudometric. By Corollary \ref{cor:uniformlyequiv} Weiss' result follows from our theorem.

 \begin{thm}[\cite{Oxtoby52}, Section~4]\label{thm:Oxtoby}
Let $x\in X$ be generic for some $\mu\in\M_T(X)$. Then the following conditions are equivalent:
\begin{enumerate}
\item $\mu\in\M_T^e(X)$,
\item For all $f\in\mathcal{C}(X)$ and $\alpha>0$
$$\dbar\bigg(\bigg\{n\,:\,\bigg|\frac{1}{k}\sum_{j=0}^{k-1}f\big(T^{n+j}(x)\big)-\int_Xf\text{d}\mu\bigg|>\alpha\bigg\}\bigg)\to 0\text{ as }k\to\infty.$$
\end{enumerate}
\end{thm}
\begin{thm}\label{thm:generic-dbar-closed}
Let $\{\mu_p\}_{p=1}^{\infty}\subset\mathcal M_T^e(X)$ and $\{x_p\}_{p=1}^{\infty}\subset X^{\infty}$ be such that $x_p$ is a generic point for $\mu_p$ for every $p\in\N$. If $x\in X$ is such that $D_B(x_p, x)\to 0$ as $p\to\infty$, then $x$ is generic for  some $\mu\in\M_T^e(X)$.
\end{thm}
\begin{proof}
By Theorem~\ref{thm:ciagloscbesic} the function $(X^{\infty}, D_B)\ni \und{x}\to\om{\und{x}}\in\left(\M(X), H_{\M(X)}\right)$ is continuous, and so there exists $\mu\in\mathcal\M_T(X)$ such that $\mu_p\to\mu$ as $p\to\infty$ and $x\in\Gen(\mu)$. We will show that $\mu\in\M_T^e(X)$. To this end we are going to show that the condition from Theorem \ref{thm:Oxtoby}(2) is fulfilled. Fix $f\in\mathcal{C}(X)$.
Notice that for every $k,n,p\in\mathbb{N}$ we have
\begin{multline*}
\left|\frac{1}{k}\sum_{j=0}^{k-1}f\big(T^{n+j}(x)\big)-\int_Xf\text{d}\mu\right|
\leq\left|\frac{1}{k}\sum_{j=0}^{k-1}f\big(T^{n+j}(x)\big)-\frac{1}{k}\sum_{j=0}^{k-1}f\big(T^{n+j}(x_p)\big)\right|\\
+\left|\frac{1}{k}\sum_{j=0}^{k-1}f\big(T^{n+j}(x_p)\big)-\int_Xf\text{d}\mu_p\right|+\left|\int_Xf\text{d}\mu_p-\int_Xf\text{d}\mu\right|.
\end{multline*}
Take any $\alpha>0$.
Let $P'\in\N$ be such that for every $p\ge P'$ we have
\[
\left|\int_Xf\text{d}\mu_p-\int_Xf\text{d}\mu\right|\le \alpha/3.
\]
Then for every $p\ge P'$ and any $k$ we have
\begin{multline*}
\bigg\{n\in\N_0\,:\,\bigg|\frac{1}{k}\sum_{j=0}^{k-1}f\big(T^{n+j}(x)\big)-\int_Xf\text{d}\mu\bigg|>\alpha\bigg\}
\subset\\
\bigg\{n\in\N_0\,:\,\frac{1}{k}\sum_{j=0}^{k-1}\bigg|f\big(T^{n+j}(x)\big)-f\big(T^{n+j}(x_p)\big)\bigg|>\alpha/3\bigg\}
\cup\\
\bigg\{n\in\N_0\,:\,\bigg|\frac{1}{k}\sum_{j=0}^{k-1}f\big(T^{n+j}(x_p)\big)-\int_Xf\text{d}\mu_p\bigg|>\alpha/3\bigg\}.
\end{multline*}
Therefore it is enough to prove that for any $\eps>0$ we can find $p\ge P'$ and $K\in\N$ such that for all $k\ge K$ the following inequalities hold
\begin{align*}
\dbar\Bigg(\bigg\{n\in\N_0\,:\,\frac{1}{k}\sum_{j=0}^{k-1}\bigg|f\big(T^{n+j}(x)\big)-f\big(T^{n+j}(x_p)\big)\bigg|>\alpha/3\bigg\}\Bigg)<\eps/2,\\
\dbar\Bigg(\bigg\{n\in\N_0\,:\,\bigg|\frac{1}{k}\sum_{j=0}^{k-1}f\big(T^{n+j}(x_p)\big)-\int_Xf\text{d}\mu_p\bigg|>\alpha/3\bigg\}\Bigg)<\eps/2.
\end{align*}
Note that the second inequality above holds for each $p$ and all sufficiently large $k$ by Theorem \ref{thm:Oxtoby}. So assume that the first inequality
fails for some parameter $\eps$, that is, there exists $\eps_0>0$ such that for all $p\ge P'$ and every $K\in\N$ there is $k_0\ge K$ with
\begin{equation}\label{eq:inequality}
\dbar\Bigg(\bigg\{n\in\N_0\,:\,\frac{1}{k_0}\sum_{j=0}^{k_0-1}\bigg|f\big(T^{n+j}(x)\big)-f\big(T^{n+j}(x_p)\big)\bigg|>\alpha/3\bigg\}\Bigg)\ge\eps_0/2.
\end{equation}
Let $\delta>0$ be such that $x',x''\in X$ and $\rho(x',x'')<\delta$ imply $|f(x')-f(x'')|<\alpha/6$. We are going to show that \eqref{eq:inequality} implies that there is a constant $\gamma>0$ such that for all $p\ge P'$ we have
\begin{equation}\label{eq:contra}
\dbar\bigg(\Big\{n\in\N_0\,:\,\rho\big(T^{n}(x),T^{n}(x_p)\big)\ge \delta\Big\}\bigg)\ge\gamma.
\end{equation}
But the inequality \eqref{eq:contra} cannot hold for all $p$, because it contradicts $D_B(x_p,x)\to 0$ as $p\to \infty$. Hence it remains to prove that \eqref{eq:inequality} implies \eqref{eq:contra}. Let $k_0$ be one of those $k$'s for which \eqref{eq:inequality} holds and let
\[
B(n)=\bigg\{0\le m <n\,:\,\frac{1}{k_0}\sum_{j=0}^{k_0-1}\bigg|f\big(T^{m+j}(x)\big)-f\big(T^{m+j}(x_p)\big)\bigg|>\alpha/3\bigg\}.
\]
It follows from \eqref{eq:inequality} that for each $N_0$ there exists $n\ge N_0$ such that
$\#B(n)\ge n\eps_0/4$. For each $m\in B(n)$ it is easy to see that
\[
\#\bigg\{0\le j <k_0\,:\,\bigg|f\big(T^{m+j}(x)\big)-f\big(T^{m+j}(x_p)\big)\bigg|>\alpha/6\bigg\}\ge \frac{\alpha k_0}{12||f||_\infty}.
\]
Therefore for infinitely many $n$ we have
\[
\#\bigg\{0\le m <n+k_0\,:\,\rho\big(T^{m}(x),T^{m}(x_p)\big)>\delta\bigg\}\ge \frac{\#B(n)}{k_0}\cdot \frac{\alpha k_0}{12||f||_\infty}\ge\frac{n\eps_0\alpha }{48||f||_\infty}.
\]
But the above inequality clearly yields \eqref{eq:contra}.
\end{proof}

\section{Weak specification implies the asymptotic average shadowing property}\label{weakspec}
In \cite{KKO} it was proved that a surjective system with the almost specification property satisfies the asymptotic average shadowing property. On the other hand it follows from \cite{KOR} that the notions of almost specification and weak specification are independent, that is there exist: a system with the weak specification property which does not have the almost specification property and a system satisfying the almost specification property without the weak specification property. In this section we will show that weak specification also implies the asymptotic average shadowing property.

Recall that $g \colon \N\times(0,\eps_0)\to\N_0$, where $\eps_0>0$, is a \emph{mistake function} if for all $\eps<\eps_0$  and all
$n \in\N$  we have $g(n, \eps) \le g(n + 1, \eps)$ and
$g(n, \eps)/n\to 0$ as $n\to \infty$.
Given a mistake function $g$ we define a function $k_g\colon (0,\infty) \to \N$ by
$k_g(\eps)=\min\{n\in\N\,:\,g(m,\eps)<m\eps\text{ for all  }m\ge n\}$.
For $0<\eps<\eps_0$ and $n\ge k_g(\eps)$ large
enough for the inequality $g(n, \eps) < n$ to hold
we set
\[
I(g; n, \eps) := \{\Lambda\subset \set{0,1,\ldots, n - 1} : \#\Lambda \ge n-g(n,\eps)\}.
\]

Given $x\in X$, $\eps>0$, $n\in\N$ and a mistake function $g$ we define a \emph{Bowen ball} by
$$B_n(x,\eps)=\big\{y\in X\,:\,\rho\big(T^j(x), T^j(y)\big)<\eps\text{ for every }0\leq j<n\big\}$$
and a \emph{Bowen ball with mistakes} by$$B_n(g;x,\eps)=\big\{y\in X\,:\,\{0\leq j<n\,:\,\rho(T^j(x),T^j(y))<\eps\}\in I(g,n,\eps)\big\}.$$

We also need the following:

\begin{lem}[\cite{KLO}, Theorem 17]\label{lem:3.2}
Let $(X,T)$ be a dynamical system. If $T$ is surjective and has the weak specification property, then $T$ is mixing.
\end{lem}
\begin{lem}\label{lem:3.4}
Let $(X,T)$ be a surjective dynamical system with the weak specification property and $M_\eps(n)$ is the constant associated to $\eps>0$ and $n\in\mathbb{N}$ in the definition of the weak specification property. Assume we are given
\begin{enumerate}
\item an~increasing sequence of integers $\{\alpha_j\}_{j=1}^{\infty}\subset\N$ with $\alpha_1=1$,
\item a~sequence of positive real numbers $\{\eps_i\}_{i=1}^{\infty}$ such that for every $i\in\N$ one has $$\eps_{\alpha_i}=\eps_{\alpha_i+1}=\ldots=\eps_{\alpha_{i+1}-1}=2\cdot\eps_{\alpha_{i+1}},$$
\item a~sequence of points $\{x_i\}_{i=1}^{\infty}\subset X$,
\item a~sequence of integers $\{n_i\}_{i=0}^{\infty}$ satisfying $n_0=0$, $M_{\eps_i/2}(n_i)\leq n_i$ and $$n_{\alpha_i}=n_{\alpha_i+1}=\ldots=n_{\alpha_{i+1}-1}\text{ for any }i\in\N.$$
\end{enumerate}
Then, setting $$l_j=\sum_{s=0}^{j-1}n_s\text{ for }j\in\N$$
we can find $z\in X$ such that for every $j\in\N$ we have
\[
T^{l_j}(z)\in\overline{B_{n_j}(g;x_j, \eps_j)},
\]
$\text{where }g\colon \N\times(0,1)\to\N \text{ is given by }g(n,\eps)=M_{\eps/2}(n)$.
\end{lem}
\begin{proof}
We will define inductively a sequence $\{z_n\}_{n=1}^{\infty}\subset X$ such that for every $i\in\N$, every $s=1,2,\ldots, i$ and every $j\in\{\alpha_s,\ldots, \alpha_{s+1}-1\}$ we have
\begin{equation}\label{lem:zj:ws}
T^{l_j+M_{\eps_j/2}(n_j)}(z_i)\in B_{n_j-M_{\eps_j/2}(n_j)}\left(T^{M_{\eps_j/2}(n_j)}(x_j),\eps_j\cdot(1-1/2^{i-s+1})\right).
\end{equation}
Observe that if we put $m=k+M_{\eps}(k)$ for some $k\in \N$ and $\eps>0$ then $M_{\eps}(m)\geq M_{\eps}(k)$.
It follows that \eqref{lem:zj:ws} implies that
\begin{equation}\label{eq:end}
T^{l_j}(z_i)\in B_{n_j}\big(g;x_j,\eps_j\cdot(1-1/2^{i-s+1})\big)\subset\overline{B_{n_j}(g;x_j, \eps_j)}.
\end{equation}

We will now define $z_1$. Take $j\in\{1,2,\ldots,\alpha_2-1\}$ and note that if we set $a_j=l_{j-1}+M_{\eps_1/2}(n_{j})$ and $b_j=l_{j-1}+n_{j}$,
then
\[
\set{T^{[a_j,b_j]}(x_j) : j=1,2,\ldots, \alpha_2-1}
\]
is an $M_{\eps_1/2}(n_1)$-spaced specification.
By the weak specification property, there exists $z_1$ satisfying
$$T^{l_j}(z_1)\in B_{n_j-M_{\eps_1/2}(n_j)}\left(T^{M_{\eps_1/2}(n_j)}(x_j),\eps_1\cdot(1-1/2)\right)$$
for each $j=1,\ldots, \alpha_{2}-1$. Indeed \eqref{lem:zj:ws} holds for $i=1$.

Assume that we have constructed a point $z_k$ satisfying \eqref{lem:zj:ws}. Then the sequence
\begin{equation}
\set{T^{[0,\alpha_{k}]}(z_k)}\cup \set{T^{[a_j,b_j]}(x_j) : j=\alpha_{k+1},\ldots, \alpha_{k+2}-1}
\label{eq:spec:zk}
\end{equation}
where $a_j=l_{j-1}+M_{\eps_{j}/2}(n_{j})$ and $b_j=l_{j-1}+n_{j}$ 
for $j=\alpha_{k+1},\ldots, \alpha_{k+2}-1$
is an $M_{\eps_{\alpha_{k+1}}/2}(n_{\alpha_{k+1}})$-spaced specification.
Using the weak specification property we find a point $z_{k+1}$ which $\eps_{\alpha_{k+1}}/2$-traces the specification given by \eqref{eq:spec:zk}.
Observe that for $s=1,\ldots, k$ and $j=\alpha_s,\ldots, \alpha_{s+1}-1$ we have
\begin{eqnarray*}
\eps_j\cdot(1-1/2^{k-s+1})+\eps_{\alpha_{k+1}}/2&=&\eps_j\cdot(1-1/2^{k-s+2})-\frac{\eps_{\alpha_s}}{2^{k-s+2}}+\eps_{\alpha_{k+1}}/2\\
&=&\eps_j\cdot(1-1/2^{(k+1)-s+1})-\frac{2^{k+1-s}\eps_{\alpha_{k+1}}}{2^{k-s+2}}+\eps_{\alpha_{k+1}}/2\\
&=&\eps_j\cdot(1-1/2^{(k+1)-s+1})
\end{eqnarray*}
hence \eqref{lem:zj:ws} is satisfied also with $i=k+1$.
Finally note that $d(z_n,z_{n+1})<\eps_{\alpha_{n+1}}$ for every $n$, hence $\set{z_n}_{n=1}^\infty$ is a Cauchy sequence.
Putting $z=\lim_{n\to \infty} z_n$ ends the proof by \eqref{eq:end}.
\end{proof}

The next theorem is an analog of \cite[Theorem~3.5]{KKO}. Although the weak specification property and the almost specification property are independent the proof follows the same lines.


\begin{thm}\label{thm:wspec_aasp}
	If a~surjective dynamical system $(X,T)$ has the weak specification property, then it has the asymptotic average shadowing property.
\end{thm}
\begin{proof}
The proof \cite[Theorem 3.5]{KKO} relies on chain mixing of every surjective dynamical system with the almost specification
property (\cite[Lemma~3.2]{KKO}) and \cite[Lemma~3.4]{KKO}. The latter result is analogous to Lemma~\ref{lem:3.4} with the only difference, that
the almost specification property is replaced by the weak specification property.
Therefore the Theorem follows from the proof of \cite[Theorem~3.5]{KKO} with Lemma~\ref{lem:3.2} and \ref{lem:3.4} used where \cite[Lemma~3.2]{KKO} and \cite[Lemma~3.4]{KKO} were applied in \cite{KKO}.
\end{proof}

\begin{rem}
In fact, it was proved in \cite[Corollary 6.9]{WOC} that there is no need to assume that $T$ is surjective 
in \cite[Theorem~3.5]{KKO}. 
We do not know whether the surjectivity assumption can be relaxed in Theorem \ref{thm:wspec_aasp}.
\end{rem}

\section{Asymptotic average pseudo-orbits and periodic decompositions}

Following Banks \cite{Banks97} we say that a collection $\mathcal D=\{D_0,\ldots,D_{n-1}\}$ of subsets of $X$ is a {\it regular periodic decomposition} if the following conditions are satisfied:
\begin{enumerate}
\item $T(D_i)\subset D_{(i+1)\,\text{mod}\,n}$ for $0\leq i<n$,
\item $D_i=\overline{\text{int}(D_i)}$ for each $0\leq i<n$,
\item the intersection $D_i\cap D_j$ is nowhere dense whenever $i\neq j$,
\item $\bigcup_{i=0}^{n-1}D_i=X$.
\end{enumerate}

Before we state the main result of this section we make the following observation. We leave the proof to the reader.

\begin{rem}\label{rem:aasp}
Fix $r>0$. Let $\{n_k\}_{k=0}^{\infty}\subset\N$ be such that $0=n_0<n_1<n_2<\ldots$ and  $n_k=0\; (\text{mod }r)$.
Then $\dbar\big(\{s_N\,:\,N\in\N\}\big)=0$, where $s_N=\sum_{j=0}^N n_j$. If $\und{z}=\{z_j\}_{j=0}^{\infty}$ is given by $z_k=T^{k-s_N}(y_N)$ for $s_N\leq k<s_{N+1}$, where $\{y_N\}_{N=0}^{\infty}\subset X$, then:
\begin{enumerate}
\item The sequence $\und{z}=\{z_n\}_{n=0}^{\infty}$ is an average asymptotic pseudo orbit for $T$ and $\und{z}'=\{z_{nr}\}_{n=0}^{\infty}$ is an average asymptotic pseudo orbit for $T^r$.
\item If $x\in X$ is such that $D_B(\und{z}',\und{x}_{T^r})=0$, 
then $D_B(\und{z},\und{x}_{T})=0$.
\end{enumerate}
\end{rem}

Now we are ready to prove the following.
\begin{thm}\label{thm:pertrace}
Assume that $(X,T)$ is a dynamical system with a regular periodic decomposition $\{D_0,\ldots, D_{r-1}\}$
and that $(D_0,T^r)$ has the asymptotic average shadowing property.
If $V\subset\MT(X)$ is nonempty, closed and connected, then there is a point $x\in X$ such that $\om{x}=V$. In particular, every $T$-invariant measure has a generic point.
\end{thm}
\begin{proof}
Let $V$ be a nonempty, closed and connected subset of $\MT(X)$. Sigmund (see \cite[Remark 1]{Sigmund77}) proved that
there is an asymptotic average pseudoorbit $\s{z}$ for $T$ such that
$\om{\underline z}=V$ (note that Sigmund used different terminology). It is easy to see that actually \cite[Remark 1]{Sigmund77}
allows us to assume that
$z_k=T^{k-s_N}(y_N)$ for $s_N\leq k<s_{N+1}$ for some sequences $\{n_k\}_{k=0}^{\infty}\subset\N$ and $\{y_N\}_{N=0}^{\infty}\subset D_0$ such that
$0=n_0<n_1<n_2<\ldots$ and $n_k=0\; (\text{mod }r)$.
By Remark~\ref{rem:aasp} the sequence $\und{z}'=\{z_{nr}\}_{n=0}^{\infty}$ is an average asymptotic pseudo-orbit for $T^r$
and furthermore, by definition we obtain that also $\und{z}'\subset D_0$.
Let $x\in D_0$ asymptotically trace on average $\und{z}'$ for $T^r$. 
Then $D_B(\und{z}',\und{x}_{T^r})=0$ and
using Remark~\ref{rem:aasp} one more time we get $D_B(\und{z},\und{x}_T)=0$.
By Corollary \ref{cor:D_B-zero} we have $\om{x}=\om{\und{z}}=V$ which ends the proof.
\end{proof}


Since it is proved in \cite{KKO} that the almost specification property implies the average shadowing property and in Section~\ref{weakspec} we have shown that the same follows from the weak specification property,
Theorem~\ref{thm:pertrace} immediately implies the following corollary, which can be also found in \cite{DTY} (with a different proof).

\begin{cor}\label{cor:weakspec}
	If $(X,T)$ is surjective and has the relative weak specification property (or relative almost specification property), then for every asymptotic average pseudo orbit $\und{x}=\{x_i\}_{i=0}^{\infty}\in X^{\infty}$ there exists a point $z$ such that 
$\om{\und x}=\om{z}$.
	In particular, if $V\subset\MT(X)$ is nonempty, closed and connected, then there is a point $z\in X$ such that $\om{z}=V$ and every $T$-invariant measure has a generic point.
\end{cor}

Corollary \ref{cor:weakspec} extends results from \cite{Dateyama81, Dateyama82,Sigmund70, Sigmund74, Sigmund76, Sigmund77}.

\begin{figure}
  \begin{minipage}[h]{.45\textwidth}
    \begin{center}
    \includegraphics[scale=1]{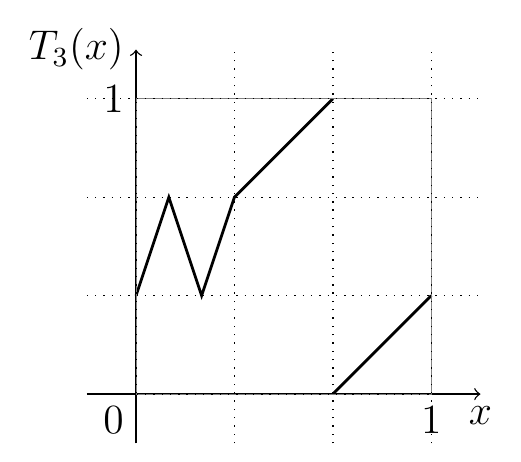}
    \caption{Graph of the map $T_3$.\label{fig:ex2}}
    \end{center}
  \end{minipage}
  \begin{minipage}[h]{.45\textwidth}
    \begin{center}
   \includegraphics[scale=1]{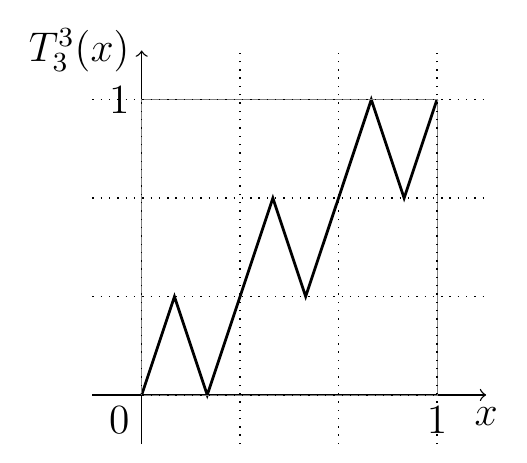}
   \caption{Graph of the map $T_3^3$.\label{fig:ex3}}
    \end{center}
  \end{minipage}
  \hfill
\end{figure}

\begin{ex}\label{ex:example}
The following example shows that there are maps which do not have the asymptotic average shadowing property but have a regular periodic decomposition $\{D_0,\ldots,D_{r-1}\}$ such that $T^r|_{D_0}$ satisfies the asymptotic average shadowing property.

By $\mathbb{T}^k=\mathbb{R}^k/\mathbb{Z}^k$ we denote the $k$-dimensional torus.
Let $S\colon \mathbb T^k\to\mathbb T^k$ be a toral automorphism whose associated linear map is nonhyperbolic
but has no roots of unity as eigenvalues. It was proved in \cite{Lind82, Marcus80} that such an automorphism satisfies the weak specification property.
Let $n\ge 2$ and $T_n\colon\mathbb T^1\to\mathbb T^1$ (see Figure \ref{fig:ex2}) be given by
\begin{equation*}
T_n(x)=\left\{\begin{array}{ll}3x+1/n,&\text{ if }0\leq x<1/(3n),\\
-3x+3/n,&\text{ if }1/(3n)\leq x<2/(3n),\\
3x-1/n,&\text{ if }2/(3n)\leq x<1/n,\\
x+1/n \mod 1, & \text{ if }1/n\leq x\leq 1.
\end{array}\right.
\end{equation*}
Define $F_n=S\times T_n\colon\mathbb T^{k+1}\to \mathbb T^{k+1}$. 
Notice that $$\big\{\mathbb T^k\times [0,1/n],\,\mathbb T^k\times [1/n,2/n],\, \ldots,\, \mathbb T^k\times [(n-1)/n,1]\big\}$$ is a regular periodic decomposition for $F_n$. Moreover, $F_n^n|_{\mathbb T^k\times[0,1/n]}$ is a product of two maps which have the weak specification property and hence satisfies it as well.  To see that $F_n$ does not have the asymptotic average shadowing property notice that the Lebesgue measure is a fully supported invariant measure for $F_n$.  By \cite[Theorems 3.7 \& 4.3]{KKO} any map with the asymptotic average shadowing and fully supported invariant measure must be weakly mixing, but $F_n^n$ is not even transitive as $F_n^n\big(\mathbb T^k\times[0,1/n]\big)\subset \mathbb T^k\times[0,1/n]$ (see Figure \ref{fig:ex3}).
\end{ex}
\section{Shift space with the asymptotic shadowing property but without almost specification}
In this section we show that the proximal dynamical system $\Y$ with positive entropy introduced in \cite{PO_BLMS} has the average shadowing property and fully supported invariant measure. By \cite[Theorem~5.3]{KOR}, a nontrivial proximal system like $\Y$ never has the almost specification property. Recall that a pair $(x,y)\in X\times X$ is \emph{proximal}, if $$\liminf\limits_{n\to\infty}\rho\big(T^n(x),T^n(y)\big)=0.$$ A dynamical system is \emph{proximal} if every pair $(x,y)\in X\times X$ is proximal. It remains open whether $\Y$ has the asymptotic average shadowing property.

We recall the construction from \cite{PO_BLMS}. Let $\Al=\set{0,\diamondsuit}$ and
let $1<s<t$ be positive integers. We define a labelled graph $\G^{(s,t)}=(V,E,i,t,\Lom)$ over $\Al$ in the following way (for the definition of labelled graph and sofic shift see \cite{ISDC}):
\begin{multline*}
V=\set{v_0,\ldots, v_{t-s},w_0,\ldots, w_{s-1}}, \quad E=\set{e_1,\ldots, e_{t-s},e_1',\ldots , e_{s-1}',\hat{e}_1,\hat{e}_2,\hat{e}_3},\\
i(e)=\begin{cases}
v_{i-1},& e=e_i\\
w_{i-1},& e=e_i'\\
w_{s-1},& e=\hat{e}_3\\
v_{t-s},& e\in \set{\hat{e}_1,\hat{e}_2},
\end{cases}
t(e)=\begin{cases}
v_{i},& e=e_i\\
w_{i},& e=e_i'\\
v_0  ,& e=\hat{e}_3\\
w_{0},& e=\hat{e}_1\\
w_{1},& e=\hat{e}_2,
\end{cases}\\
\Lom(e)=\begin{cases}
\diamondsuit,& e=e_i\text{ or }e=\hat{e}_1\\
0,& e=e_i'\text{ or }e\in \set{\hat{e}_2,\hat{e}_3}.
\end{cases}
\end{multline*}

\begin{figure}
\includegraphics{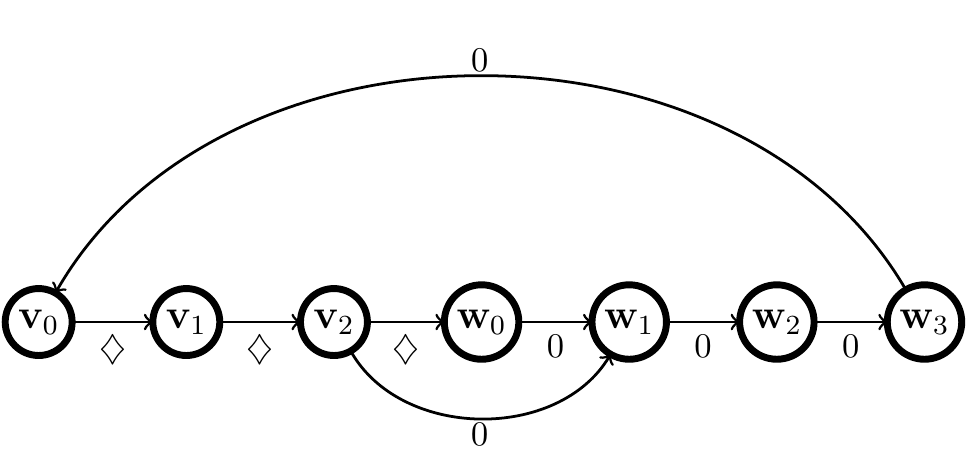}
\caption{The graph $\mathcal{G}^{(4,6)}$}
\label{graph}
\end{figure}
Denote by $X_{{(s,t)}}$ the shift space consisting of sequences over $\Al$ which are labels of infinite paths on $\G^{(s,t)}$.
In other words, $x\in X_{{(s,t)}}$ is a concatenation of blocks of $s$ zeros separated by blocks of $t-s$ or $t-s+1$ symbols $\diamondsuit$.
Denote by $p^{(s,t)}$ a periodic sequence of the form $(\diamondsuit^{t-s} 0^s)^\infty\in  \Al^\infty$. Now, define the following subsets of $\set{0,1}^\infty$:
\begin{eqnarray*}
F^{(s,t)}& = & \set{x\in \set{0,1}^\infty \; : \; \text{there exists}\; y\in X_{{(s,t)}} \textrm{ such that if } y_i=0, \textrm{ then } x_i=0},\\
P^{(s,t)}&=& \set{x\in \set{0,1}^\infty \; : \text{ if }p^{(s,t)}_i=0,\text{ then }x_i=0}.\\
\end{eqnarray*}
It is easy to verify that each $F^{(s,t)}$ is a shift in $\set{0,1}^\infty$.

Let $s_i=2^i$ and $t_i=10^i$. 
Define
$$
\Y=\bigcap_{i=1}^\infty F^{(s_i,t_i)}.
$$
Note that
\begin{equation}\label{eq:ass}
t_i> 3 s_i+2i > 5i \text{ and  }0 < \sum_{i=1}^\infty \frac{s_i}{t_i}= \frac{1}{4}<1.
\end{equation}
It was proved in \cite{PO_BLMS} that \eqref{eq:ass} implies that
$\Y$ is a topologically mixing and proximal system with positive topological entropy.
Let
\[
Y_n=X^{(s_1,t_1)}\cap\ldots \cap X^{(s_n,t_n)}.
\]

\begin{lem}\label{lem:dbar-close}
For every $\eps>0$ there is $n>0$ such that if $x\in Y_n$, then there is $z\in \Y$ with
$\dbar(x,z)<\eps$.
\end{lem}
\begin{proof}
Fix $\eps>0$. Use \eqref{eq:ass} to find $n>0$ such that
\[
\sum_{i=n+1}^\infty \frac{s_i}{t_i}<\eps.
\]
Take $x\in Y_n$. Define $z$ by $z_j=0$ if there is $i>n$ such that $p^{(s_i,t_i)}_j=0$,  and $z_j=x_j$ otherwise. Then $z\in \Y$ and
$\dbar(x,z)<\eps$.
\end{proof}


\begin{lem}\label{lem:dec} Let $X$ be a shift space.
If for every $\eps>0$ there is a shift space $X_\eps$ with the average shadowing property such that $X\subset X_\eps$ and for every $x\in X_\eps$ there is $z\in X$ with
$\dbar(x,z)<\eps$, then $X$ has the average shadowing property.
\end{lem}
\begin{proof}
Fix $\eps>0$. Let $\beta>0$ be such that $\dbar(x,x')<\beta$ implies $D_B(x,x')<\eps/2$ (the existence of such a $\beta$ is guaranteed by Corollary~\ref{cor:uniformlyequiv}).
Consider shift space $X_{\beta}$. Use the average shadowing property for $X_{\beta}$ to find $\delta>0$ such that every
$\delta$-average pseudoorbit for $\sigma$ contained in $X_{\beta}$ is $\eps/2$ traced on average.
Let $\und{x}$ be a $\delta$-average pseudoorbit for $\sigma$ contained in $X$. Since $X\subset X_{\beta}$ there is a point $y\in X_{\beta}$
which $\eps/2$-traces $\und{x}$ in average. Use the assumption to find a point $z\in X$ such that $\dbar(y,z)<\beta$.
Hence $D_B(\und{z}_T,\und{x})\le D_B(z,y)+D_B(\und{y}_T,\und{x})<\eps$.
\end{proof}

\begin{lem}\label{lem:ret_times}
For every nonempty open set $U\subset \Y$ there is $y\in \Y$ such that $N(y,U)$ has positive upper density.
\end{lem}
\begin{proof}
Let $z\in U$ and $k$ be such that the cylinder set defined by $z_{[0,k)}$ is contained in $U$. Take any  $n>k$ such that $t_n-s_n>k$ and
\[
1-k\sum_{r=n+1}^\infty \frac{s_r}{t_r}:=\gamma>0
\]
Then the periodic point $x=(z_{[0,k)}0^{t_n-k})^\infty$ belongs to $Z=\bigcap_{j=1}^n F^{(s_i,t_i)}$
and $t_n\N \subset \set{i : x_{[i,i+k)}=z_{[0,k)}}$.
Define $p\in \{0,\diamondsuit\}^\infty$ by
$p_j=0$ if there is $i>n$ such that $p^{(s_i,t_i)}_j=0$,  and $p_j=\diamondsuit$ otherwise.
Then $p\in \bigcap_{i=n+1}^\infty P^{(s_i,t_i)}$ and we claim that 
the set $A=\set{i : p_{[i,i+k)}=\diamondsuit^k}$
has positive upper density. To see this, note that
\[
d(\{i:p_i=0\})\le \sum_{r=n+1}^\infty \frac{s_r}{t_r},
\]
and since $p_i=0$ and $i \ge k$ implies $p_{[i-l,i-l+k)}\neq\diamondsuit^k$ for $l=-k+1,-k+2,\ldots,0$ we have
\[
\dbar(A)\ge 1- k d(\{i:p_i=0\})\ge \gamma.
\]
We have $\bigcup_{j=0}^{t_n-1}j+t_n\N=\N$, therefore for some $j\in\{0,1,\ldots,t_n-1\}$ the set $A_j = (A+j) \cap t_n\N$ has upper density at least $\gamma/t_n$. For that $j$ we construct a point $y=y_0y_1y_2\ldots$ by
$$
y_i=\begin{cases}
x_i,& i\in A_j+[0,k),\\
0 ,& \text{otherwise}.
\end{cases}
$$
then $y\in \Y$ and obviously by our construction we have that $A_j\subset N(y,U)$, hence $N(y,U)$ has positive upper density.
\end{proof}

\begin{thm}\label{thm:YproxFSM}
The shift space $\Y$ is mixing, has the average shadowing property, positive topological entropy and fully supported invariant measure.
\end{thm}
\begin{proof}
Proximality, positive topological entropy and topological mixing follows from \cite{PO_BLMS}. Let $\mathcal{U}=\{U_n:n\in\N\}$ be the base for the topology on $\Y$.
Lemma~\ref{lem:ret_times} implies that for any open set $U_n\in\mathcal{U}$ we can find an invariant measure $\mu_n$ such that $\mu(U)>0$ (e.g. see \cite{Fur}).
Then $\mu=\sum_{n=1}^\infty 2^{-n}\mu_n$ is a fully supported measure.

Note that $Z_N=\bigcap_{i=1}^N F^{(s_i,t_i)}$ is the intersection of mixing sofic shifts for every $N\in\N$, therefore $Z_N$ is a mixing sofic shift.
Since every mixing sofic shift has the specification property, it follows from \cite{KO_AASP} that $Z_N$ has the asymptotic average shadowing property and so it also has the average shadowing property by \cite{KKO}. Now $\Y$ has the average shadowing property by Lemmas \ref{lem:dbar-close} and \ref{lem:dec}.
\end{proof}
\begin{cor}
The subshift $\Y$ has the average shadowing property but does not have the almost specification property.
\end{cor}
\begin{proof}
By \cite{KOR} minimal points are dense in the measure center of any dynamical system with the almost specification property.
On the other hand, it follows from Theorem~\ref{thm:YproxFSM} that $\Y$ is a nontrivial (hence uncountable) proximal system whose measure center equals the whole space. Proximality implies that $\Y$ has a fixed point which is the unique minimal point in the system (e.g. see \cite{AK}). Therefore minimal points of $\Y$ can not be dense in the measure center and $\Y$ can not have the almost specification property.
\end{proof}

Our last example shows that the assumptions of Lemma \ref{lem:dec} can not be weakened.

Given a shift space $X$ over an alphabet $\mathcal{A}$ we define the \emph{finite type approximation of order $m\in\mathbb{N}$} of $X$ by
\[
X_m=\{x\in\mathcal{A}^\infty:\,\forall k\in\mathbb{N}_0\,\exists y\in X\,x_{[k,k+m]}=y_{[0,m]}\}.
\]
In other words $X_m$ is a shift of finite type determined by finite blocks of length $m+1$ occurring in $X$ (see \cite[p.111]{DGSbook}).
The following lemma is probably well-known (cf. Proposition 3.62 in \cite{Kurkabook03}).
\begin{lem}\label{lem:equivalence}
Let $X$ be a mixing shift space. The following conditions are equivalent:
\begin{enumerate}
  \item $X$ is chain mixing,
  \item $X_m$ is mixing for every $m\in\mathbb{N}$.
\end{enumerate}
\end{lem}
\begin{proof}[Sketch of the proof] Note that for every $m$ the points in $X_m$ are in the one to one correspondence with $\delta_m$-chains in $X$, where $\delta_m$ goes to $0$ as $m\to\infty$. To see this take a point $x\in X_m$ and for each $j\in\mathbb{N}_0$ pick $x^{(j)}$ in the cylinder in $X$ determined by $x_{[j,j+m]}$. Furthermore, chain mixing is equivalent to mixing for systems with the shadowing property and all shifts of finite type have the shadowing property by \cite{WaltersSFT}.
\end{proof}

\begin{cor}\label{lem:approx}
Let $X$ be a mixing shift space. Then there exists a decreasing sequence $X_0\supset X_1\supset \ldots \supset X$ of mixing
shifts of finite type such that $X=\bigcap_{n=1}^\infty X_n$.
\end{cor}
Note that if $X=\bigcap_{n=1}^\infty X_n$ and $X_1\supset X_2\supset \ldots \supset X$ then for every $\eps>0$ there is $n$ such that $\dist(x,X)<\eps$ for each $x\in X_n$.
Next Theorem shows that this is not enough to imply the average shadowing property on the intersection and
the condition $\inf\{\dbar(x,y):x\in X,\,y\in X_n\}<\eps$ for all $n$ large enough in Lemma~\ref{lem:dec} is essential.

\begin{thm}
There exists a decreasing sequence $X_1\supset X_2\supset \ldots \supset X$ of mixing
sofic shifts such that the shift space $X=\bigcap_{n=1}^\infty X_n$ is mixing but does not have the average shadowing property.
\end{thm}
\begin{proof}
Consider the shift space $X$ constructed in Section~6 of \cite{FKKL}. By \cite[Lemmas~6--8]{FKKL} $X$ has the following properties:
\begin{enumerate}
\item periodic points are dense,
\item $X$ is mixing, nontrivial and $0^\infty\in X$,
\item if $x\in X$ is not periodic then $d(\{i: x_i=1\})=0$.
\end{enumerate}
Since periodic points are dense and $X$ is nontrivial, there exists a periodic point $p\in X$ such that $d(\{i : p_i=1\})=\alpha>0$. Using $p$ and $0^\infty$ interchangeably
it is not hard to construct an asymptotic average pseudo-orbit $\{y^{(i)}\}_{i=0}^\infty$ such that
\begin{enumerate}[(i)]
\item\label{ex:c1} $\dbar(\{i : y^{(i)}_0=1\})=\alpha>0$  and
\item\label{ex:c2} $\dbar(\{i : y^{(i)}_0=0\})=1$.
\end{enumerate}
If $X$ has the average shadowing property, then by \cite[Thm. 5.5]{WOC} there exists $z\in X$ such that
\[
\underline{d}(\{ i : z_i=y^{(i)}_0\})>1-\alpha/3.
\]
This implies by \eqref{ex:c1} that $\dbar(\{i : z_i=1\})\geq \alpha/2$ and so $z$ must be a periodic point with $d(\{i : z_i=1\})=\beta\geq \alpha/2$.
On the other hand by \eqref{ex:c2} we see that $\overline{d}(\{ i : z_i=0\})\geq 1-\alpha/3>1-\beta$ which is a contradiction. This shows that $X$ cannot have the average shadowing property. The proof is finished by Corollary~\ref{lem:approx}.
\end{proof}

\section*{Acknowledgements} 
We would like to thank both the reviewers for their contribution and comments. 
Dominik Kwietniak was supported by the National Science Centre (NCN) under
grant 2013/08/A/ST1/00275 and partially supported by CAPES/Brazil grant no. 88881.064927/2014-01. Martha \L{}\k{a}cka was supported by grant no. PSP: K/DSC/002690. Research of Piotr Oprocha was partly supported by the project ``LQ1602 IT4Innovations excellence in science''.


\def\ocirc#1{\ifmmode\setbox0=\hbox{$#1$}\dimen0=\ht0 \advance\dimen0
	by1pt\rlap{\hbox to\wd0{\hss\raise\dimen0
			\hbox{\hskip.2em$\scriptscriptstyle\circ$}\hss}}#1\else {\accent"17 #1}\fi}

\end{document}